\documentclass{article}

\usepackage{fancyhdr}
\usepackage{setspace, amsmath, amsthm, amssymb, amsfonts, amscd, epic, graphicx, ulem, dsfont}
\usepackage{amsthm}
\usepackage{arxiv}

\usepackage[utf8]{inputenc} 
\usepackage[T1]{fontenc}    
\usepackage{hyperref}       
\usepackage{url}            
\usepackage{booktabs}       
\usepackage{amsfonts}       
\usepackage{nicefrac}       
\usepackage{microtype}      
\usepackage{lipsum}
\usepackage{graphicx}
\graphicspath{ {./images/} }
\newtheorem{theorem}{Theorem}

\newtheorem{corollary}[theorem]{Corollary}
\newtheorem{example}[theorem]{Example}

\newtheorem{remark}[theorem]{Remark}

\title{$p$-Biharmonic hypersurfaces in Einstein space and conformally flat space}

\author{Khadidja Mouffoki \\
Mascara University, Faculty of Exact Sciences, Mascara 29000, Algeria\\
\texttt{khadidja.mouffoki@univ-mascara.dz}
\And Ahmed Mohammed Cherif\\
Mascara University, Faculty of Exact Sciences, Mascara 29000, Algeria\\
\texttt{a.mohammedcherif@univ-mascara.dz}}

\begin{document}

\maketitle

\begin{abstract}
In this paper, we present some new properties for $p$-biharmonic hypersurfaces in Riemannian manifold. We also
characterize the $p$-biharmonic submanifolds in an Einstein space. We construct a new example of proper $p$-biharmonic hypersurfaces. We
present some open problems.
\\[2mm] {\it AMS Mathematics  Subject Classification $(2020)$}: 53C43, 58E20, 53C25.
\\[1mm] {\it Key words and phrases:} $p$-biharmonic maps, $p$-biharmonic submanifolds, Einstein space.
\end{abstract}

\maketitle

%
%
\section{Introduction}
Let $\varphi:(M^m,g)\longrightarrow (N^n,h)$ be a smooth map  between Riemannian manifolds.
The $p$-energy functional of $\varphi$ is defined by
\begin{equation}\label{eq1.1}
E_{p}(\varphi;D)=\frac{1}{p}\int_{D}|d\varphi|^pv_{g},
\end{equation}
where $D$ is a compact domain in $M$, $|d\varphi|$  the Hilbert-Schmidt norm of the differential $d\varphi$,  $v^g$  the volume element on $(M^m,g)$, and $p\geq2$.\\
A smooth map is called $p$-harmonic if it is a critical point of the $p$-energy functional (\ref{eq1.1}). We have
\begin{equation}\label{eq1.2}
    \frac{d}{dt}E_{p}(\varphi_{t};D)\Big|_{t=0}=-\int_{D}h(\tau_{p}(\varphi),v)v_{g},
\end{equation}
where  $\{\varphi_{t}\}_{t\in (-\epsilon,\epsilon)}$ is a smooth variation of $\varphi$ supported in $D$,
$\displaystyle v=\frac{\partial \varphi_{t}}{\partial t}\Big|_{t=0}$ the variation vector field of $\varphi$,
and $\tau_{p}(\varphi)=\operatorname{div}^M(|d\varphi|^{p-2}d\varphi)$  the $p$-tension field of $\varphi$.\\
Let $\nabla^{M}$  the Levi-Civita connection of $(M^m,g)$, and $\nabla^{\varphi}$ the pull-back connection on $\varphi^{-1}TN$, the map $\varphi$ is $p$-harmonic if and only if (see \cite{BG,BI,ali})
\begin{equation}\label{eq1.3}
   |d\varphi|^{p-2}\tau(\varphi)+(p-2)|d\varphi|^{p-3} d\varphi(\operatorname{grad}^M|d\varphi|)=0,
\end{equation}
where $\tau(\varphi)=\operatorname{trace}_g\nabla d\varphi$ is the tension field of $\varphi$ (see \cite{BW,ES}). The $p$-bienergy functional of $\varphi$ is defined by
\begin{equation}\label{eq1.4}
    E_{2,p}(\varphi;D)=\frac{1}{2}\int_D|\tau_p(\varphi)|^2 v^g.
\end{equation}
We say that $\varphi$ is a $p$-biharmonic map if it is a critical point of the $p$-bienergy functional \eqref{eq1.4}, the Euler-Lagrange equation of the $p$-bienergy functional is given by (see \cite{cherif2})
\begin{eqnarray}\label{eq1.7}
\tau_{2,p}(\varphi)
   &=&\nonumber -|d\varphi|^{p-2}\operatorname{trace}_gR^{N}(\tau_{p}(\varphi),d\varphi)d\varphi
       -\operatorname{trace}_g\nabla^\varphi |d\varphi|^{p-2} \nabla^\varphi \tau_{p}(\varphi)\\
    &&-(p-2)\operatorname{trace}_g\nabla <\nabla^\varphi\tau_{p}(\varphi),d\varphi>|d\varphi|^{p-4}d\varphi=0,
\end{eqnarray}
where $R^N$ is the curvature tensor of $(N^n,h)$ defined by
\begin{equation*}
    R^N(X,Y)Z=\nabla^N_X \nabla^N_Y Z-\nabla^N_Y \nabla^N_X Z-\nabla^N_{[X,Y]}Z,\quad\forall X,Y,Z\in\Gamma(TN),
\end{equation*}
and $\nabla^N$ the Levi-Civita connection of $(N^n,h)$. The $p$-energy functional (resp. $p$-bienergy functional) includes as a special case $(p = 2)$ the energy functional (resp. bienergy functional), whose critical points are the usual harmonic maps (resp. biharmonic maps \cite{Jiang}).\\
A submanifold in a Riemannian manifold is called a $p$-harmonic submanifold (resp. $p$-biharmonic submanifold) if the isometric immersion defining the submanifold is a $p$-harmonic map (resp. $p$-biharmonic map). Will call proper $p$-biharmonic submanifolds a $p$-biharmonic submanifols which is non $p$-harmonic.
%
\section{Main Results}
Let $(M^m,g)$ be a hypersurface of $(N^{m+1},\langle ,\rangle )$, and  $\mathbf{i} : (M^m,g) \hookrightarrow (N^{m+1},\langle ,\rangle ) $ the canonical inclusion.
We denote by $\nabla^M$ (resp. $\nabla^N$) the  Levi-Civita connection of $(M^m,g)$ (resp. of $(N^{m+1},\langle  ,\rangle  )$),  $\operatorname{grad}^M$ (resp. $\operatorname{grad}^N$) the gradient operator in $(M^m,g)$ (resp. in $(N^{m+1},\langle ,\rangle  )$, $B$ the second fundamental form of the hypersurface  $(M^m,g)$, $A$ the shape operator with respect to the unit normal vector field $\eta$,  $H$ the mean curvature of $(M^m,g)$,  $\nabla^\perp$ the normal connection of $(M^m,g)$,  and by $\Delta$ (resp. $\Delta^{\perp}$) the Laplacian on $(M^m,g)$ (resp. on the normal bundle of $(M^m,g)$ in $(N^{m+1},\langle ,\rangle )$ (see \cite{BW,ON,YX}).
Under the notation above we have the following results.

\begin{theorem}\label{th1}
The hypersurface $(M^m,g)$ with the mean curvature vector $H= f \eta $ is $p$-bihamronic if and only if
   \begin{equation}\label{sys1}
\left\{
\begin{array}{lll}
-\Delta^M(f)  + f |A|^2 -f \operatorname{Ric}^N(\eta , \eta ) + m(p-2)   f^3 &=& 0; \\\\
 2A(\operatorname{grad}^M f) -2 f (\operatorname{Ricci}^N \eta)^\top + ( p-2 + \dfrac{m}{2} ) \operatorname{grad}^M f^2 &=& 0,
\end{array}
\right.
\end{equation}
where $\operatorname{Ric}^N $ (resp. $\operatorname{Ricci}^N$) is the Ricci curvature (resp. Ricci tensor) of $(N^{m+1},\langle ,\rangle )$.
\end{theorem}

\begin{proof}
Choose a normal orthonormal frame $\{ e_i \}_{i=1,... , m }$ on $(M^m,g)$ at $x$, so that $\{ e_i , \eta \}_{i=1,...,m} $ is an orthonormal frame on the ambient space $(N^{m+1},\langle ,\rangle )$.
Note that,  $d\mathbf{i}(X)=X$, $\nabla^{\mathbf{i}}_X Y = \nabla^N_X Y $, and the $p$-tension field of $\mathbf{i}$ is given by $ \tau_p(\mathbf{i}) = m^{\frac{p}{2} } f \eta $.
We compute the $p$-bitension field of $\mathbf{i}$
\begin{eqnarray}\label{eq2.2}
\nonumber  \tau_{2,p}(\mathbf{i})& =&  -|d\mathbf{i}|^{p-2}  \operatorname{trace}_g R^N(\tau_p(\mathbf{i}) , d\mathbf{i} ) d\mathbf{i}  \\
\nonumber             & & -(p-2) \operatorname{trace}_g \nabla \langle  \nabla^\mathbf{i} \tau_p(\mathbf{i}) , d\mathbf{i} \rangle  |d\mathbf{i}|^{p-4} d\mathbf{i}  \\
                      & &- \operatorname{trace}_g \nabla^\mathbf{i}|d\mathbf{i}|^{p-2} \nabla^\mathbf{i}\tau_p(\mathbf{i}).
\end{eqnarray}
The first term of (\ref{eq2.2}) is given by
\begin{eqnarray}\label{eq2.3}
\nonumber -|d\mathbf{i}|^{p-2} \operatorname{trace}_g R^N (\tau_p(\mathbf{i}) , d\mathbf{i} )d\mathbf{i}
                                                      &=& -|d\mathbf{i}|^{p-2} \sum_{i=1}^mR^N(\tau_p(\mathbf{i}) , d\mathbf{i}(e_i))d\mathbf{i}(e_i)\\
\nonumber                                             &=& - m^{p-1} f \sum_{i=1}^m R^N(\eta , e_i)e_i \\
\nonumber                                             &=& -m^{p-1} f \operatorname{Ricci}^N \eta \\
\nonumber                                            &=& -m^{p-1} f \left[ (\operatorname{Ricci}^N \eta )^\perp +(\operatorname{Ricci}^N \eta)^\top \right].\\
\end{eqnarray}
We compute the second term of (\ref{eq2.2})
\begin{eqnarray*}
  -(p-2)\operatorname{trace}_g \nabla\langle  \nabla^\mathbf{i} \tau_p(\mathbf{i}) , d\mathbf{i} \rangle  |d\mathbf{i}|^{p-4}d\mathbf{i}
&=& -(p-2)m^{p-2} \sum_{i,j=1}^m\nabla^N_{e_j} \langle  \nabla^N_{e_i}  f \eta , e_i \rangle  e_j,
\end{eqnarray*}
\begin{eqnarray*}
\sum_{i=1}^m\langle  \nabla^N_{e_i}  f \eta , e_i \rangle
                                           &=& \sum_{i=1}^m\left[\langle e_i(f) \eta , e_i \rangle  + f\langle \nabla^N_{e_i} \eta , e_i \rangle  \right]\\
                                  &=& -f \sum_{i=1}^m \langle \eta , B(e_i , e_i) \rangle  \\
                                           &=& -m f^2.
\end{eqnarray*}
By the last two equations, we have the following
\begin{equation}\label{eq2.4}
-(p-2)\operatorname{trace}_g \nabla\langle  \nabla^\mathbf{i} \tau_p(\mathbf{i}) , d\mathbf{i} \rangle  |d\mathbf{i}|^{p-4}d\mathbf{i}
= m^{p-1}(p-2) \left( \operatorname{grad}^M f^2 + m f^3\eta\right).
\end{equation}
The third term of (\ref{eq2.2}) is given by
\begin{eqnarray}\label{eq2.5}
\nonumber - \operatorname{trace}_g \nabla^{\mathbf{i}}|d\mathbf{i}|^{p-2} \nabla^{\mathbf{i}}\tau_p(\mathbf{i})
                                                         &=& -m^{p-1} \sum_{i=1}^m\nabla^N_{e_i} \nabla^N_{e_i} f \eta \\
\nonumber                                                &=& -m^{p-1} \sum_{i=1}^m\nabla^N_{e_i}[e_i(f) \eta +f \nabla^N_{e_i} \eta ] \\
\nonumber                                                &=& -m^{p-1}\left[ \Delta^M (f) \eta + 2 \nabla^N_{\operatorname{grad}^M f } \eta + f \sum_{i=1}^m\nabla^N_{e_i}\nabla^N_{e_i} \eta \right].
\nonumber  \\
\end{eqnarray}
Thus, at $x$, we obtain
 \begin{eqnarray} \label{eq2.6}
 \sum_{i=1}^m\nabla_{e_i}^N \nabla_{e_i}^N \eta
                                              &=& \nonumber \sum_{i=1}^m\nabla_{e_i}^N\left[(\nabla_{e_i}^N \eta)^\perp +(\nabla_{e_i}^N \eta)^\top \right] \\
                                              &=&\nonumber  -\sum_{i=1}^m\nabla_{e_i}^NA(e_i) \\
                                              &=& - \sum_{i=1}^m\nabla_{e_i}^M A(e_i)-\sum_{i=1}^mB(e_i , A(e_i)).
 \end{eqnarray}
 Since $\langle  A(X),Y\rangle  = \langle B(X,Y) , \eta\rangle  $ for all $X,Y \in \Gamma(TM) $, we get
 \begin{eqnarray}\label{eq2.7}
 \nonumber \sum_{i=1}^m\nabla_{e_i}^M A(e_i)
                                 &=&\sum_{i,j=1}^m \langle \nabla_{e_i}^M A(e_i), e_j\rangle  e_j \\
 \nonumber                       &=&\sum_{i,j=1}^m\left[ e_i \langle A(e_i) , e_j \rangle  e_j - \langle A(e_i) , \nabla^M_{e_i} e_j \rangle e_j \right]\\
 \nonumber                       &=&\sum_{i,j=1}^m e_i \langle B(e_i,e_j) , \eta \rangle  e_j \\
 \nonumber                       &=&\sum_{i,j=1}^m e_i \langle \nabla_{e_j}^Ne_i , \eta \rangle  e_j\\
                       &=&\sum_{i,j=1}^m \langle \nabla_{e_i}^N\nabla_{e_j}^Ne_i , \eta \rangle e_j.
 \end{eqnarray}
By using the definition of curvature tensor of $(N^{m+1},\langle ,\rangle)$, we conclude
\begin{eqnarray} \label{eq2.8}
\nonumber \sum_{i=1}^m\nabla_{e_i}^M A(e_i)
&=&\sum_{i,j=1}^m \left[\langle R^N (e_i,e_j)e_i , \eta\rangle e_j + \langle \nabla_{e_j}^N\nabla_{e_i}^Ne_i , \eta\rangle e_j\right] \\
\nonumber
&=&\sum_{i,j=1}^m \left[ -\langle R^N (\eta , e_i) e_i ,e_j \rangle e_j +  \langle \nabla_{e_j}^N\nabla_{e_i}^Ne_i , \eta\rangle e_j \right]\\
\nonumber
&=&  - \sum_{j=1}^m\langle \operatorname{Ricci}^N \eta , e_j \rangle e_j + \sum_{i,j=1}^m e_j\langle \nabla_{e_i}^Ne_i , \eta\rangle e_j-\sum_{i,j=1}^m \langle \nabla_{e_i}^N{e_i}, \nabla_{e_i}^N \eta \rangle e_j \\
&=& -( \operatorname{Ricci}^N \eta )^\top + m \operatorname{grad}^M f.
\end{eqnarray}
On the other hand, we have
 \begin{eqnarray} \label{eq2.9}
 \nonumber \sum_{i=1}^mB(e_i,A(e_i )) &=& \sum_{i=1}^m\langle B(e_i , A(e_i)) , \eta\rangle   \eta \\
 \nonumber                &=& \sum_{i=1}^m\langle  A(e_i) , A(e_i )\rangle  \eta \\
                          &=& |A|^2 \eta.
 \end{eqnarray}
Substituting (\ref{eq2.6}), (\ref{eq2.8}) and (\ref{eq2.9}) in (\ref{eq2.5}), we obtain
 \begin{eqnarray}\label{eq2.10}
  - \operatorname{trace}_g \nabla^{\mathbf{i}}|d\mathbf{i}|^{p-2} \nabla^{\mathbf{i}}\tau_p(\mathbf{i})
  &=&\nonumber - m^{p-1} \big[ \Delta^M(f)\eta-2 A (\operatorname{grad}^M f) + f (\operatorname{Ricci}^N \eta )^\top \\
  &&- \dfrac{m}{2} \operatorname{grad}^M f^2 - f |A|^2 \eta  \big].
 \end{eqnarray}
The Theorem \ref{th1} follows by (\ref{eq2.2})-(\ref{eq2.4}), and (\ref{eq2.10}).
\end{proof}

As an immediate consequence of Theorem \ref{th1} we have.
\begin{corollary}\label{corollary1}
A hypersurface $(M^m,g) $ in an Einstein space $(N^{m+1},\langle ,\rangle  )$ is $p$-biharmonic if and only if it's mean curvature function $f$ is a solution of the following PDEs
\begin{equation}\label{sys2}
\left\{
\begin{array}{lll}
-\Delta^M(f)  + f |A|^2  + m(p-2) f^3 - \frac{S}{m+1}f &=& 0; \\\\
 2A(\operatorname{grad}^M f) + ( p-2 + \dfrac{m}{2} ) \operatorname{grad}^M f^2 &=& 0,
\end{array}
\right.
\end{equation}
where $S$ is the scalar curvature of the ambient space.

\end{corollary}
 \begin{proof}
  It is well known that if $(N^{m+1} ,\langle  , \rangle  ) $ is an Einstein manifold then $\operatorname{Ric}^N(X,Y) = \lambda \langle X,Y\rangle $
  for some constant $\lambda$, for any $X,Y \in \Gamma(TN)$. So that
  \begin{eqnarray*}
  S&=& \operatorname{trace}_{\langle  , \rangle} \operatorname{Ric}^N\\
   &=& \sum_{i=1}^m\operatorname{Ric}^N(e_i,e_i) + \operatorname{Ric}^N(\eta, \eta) \\
   &=& \lambda (m+1),
  \end{eqnarray*}
where $\{ e_i \}_{i=1,... , m }$ is a normal orthonormal frame on $(M^m,g)$ at $x$.
Since $\operatorname{Ric}^N(\eta, \eta) = \lambda  $, on conclude that
 $$\operatorname{Ric}^N(\eta, \eta) = \frac{S}{m+1}. $$
On the other hand, we have
 \begin{eqnarray*}
 (\operatorname{Ricci}^N\eta)^{\top}&=& \sum_{i=1}^m\langle \operatorname{Ricci}^N\eta , e_i\rangle e_i \\
                                    &=& \sum_{i=1}^m\operatorname{Ric}^N(\eta , e_i)e_i\\
                                    &=& \sum_{i=1}^m\lambda\langle \eta , e_i\rangle  e_i \\
                                    &=& 0.
 \end{eqnarray*}
 The Corollary \ref{corollary1} follows by Theorem \ref{th1}.
 \end{proof}
 \begin{theorem} \label{th2}
A totally umbilical hypersurface $(M^m,g) $ in an Einstein space $(N^{m+1},\langle ,\rangle  )$ with non-positive scalar curvature is $p$-biharmonic if and only if it is minimal.
  \end{theorem}
 \begin{proof}
 Take an orthonormal frame $\{ e_i, \eta  \}_{i=1,...,m}$ on the ambient space $(N^{m+1} , \langle ,\rangle  )$ such that
 $\{ e_i \}_{i=1,...,m}$ is an orthonormal frame on $(M^m,g)$. We have
 \begin{eqnarray*}
 f &=& \langle H,\eta\rangle  \\
   &=& \frac{1}{m} \sum_{i=1}^m\langle B(e_i, e_i), \eta \rangle  \\
   &=& \frac{1}{m} \sum_{i=1}^m\langle  g(e_i, e_i) \beta \eta , \eta \rangle  \\
   &=& \beta,
 \end{eqnarray*}
 where $\beta\in C^\infty(M)$.  The $p$-biharmonic hypersurface equation $(\ref{sys2})$ becomes
\begin{equation}\label{sys3}
\left\{
\begin{array}{lll}
-\Delta^M(\beta)  + m (p-1)\beta^3 - \frac{S}{m+1} \beta &=& 0; \\\\
 (p-1 + \frac{m}{2} ) \beta \operatorname{grad}^M \beta &=& 0,
\end{array}
\right.
\end{equation}
Solving the last system, we have $\beta= 0$ and hence $f=0$, or
$$\beta= \pm\sqrt{\frac{S}{m(m+1)(p-1)}},$$
it's constant and this happens only if $S\geq0$. The proof is complete.
 \end{proof}

\section{$p$-biharmonic hypersurface in conformally flat space }
Let $ \mathbf{i}  : M^m \hookrightarrow \mathbb{R}^{m+1} $ be a minimal hypersurface with the unit normal vector field $\eta$,
$ \widetilde{ \mathbf{i} } : (M^m , \widetilde{g}) \hookrightarrow (\mathbb{R}^{m+1} , \widetilde{h} = e^{2\gamma} h ) $, $x\longmapsto
\widetilde{ \mathbf{i} }(x)=\mathbf{i}(x)=x$, where $\gamma \in C^{\infty}(\mathbb{R}^{m+1} )$,  $h= \langle ,\rangle _{\mathbb{R}^{m+1}}  $, and $\widetilde{g} $ is the induced metric by $\widetilde{h}$, that is
$$ \widetilde{g}(X,Y ) = e^{2\gamma} g(X,Y)= e^{2\gamma} \langle X,Y\rangle _{\mathbb{R}^{m+1}} ,$$ where $g$ is the induced metric by $h$.
Let $\{e_i, \eta \}_{i=1,...,m} $ be an orthonormal frame adapted to the $p$-harmonic hypersurface on $(\mathbb{R}^{m+1} , h )$, thus $ \{\widetilde{e}_i, \widetilde{\eta} \}_{i=1,...,m} $ becomes an orthonormal frame on $ (\mathbb{R}^{m+1} , \widetilde{h} ) $, where $\widetilde{e}_i=e^{-\gamma}e_i$ for all $i=1,..., m $, and $ \widetilde{\eta}=e^{-\gamma} \eta $.

\begin{theorem}\label{th3}
 The hypersurface $(M^m , \widetilde{g})$ in the conformally flat space $(\mathbb{R}^{m+1} , \widetilde{h}) $ is $p$-biharmonic if and only if
\begin{equation}\label{sys4}
\left\{
\begin{array}{lll}
 \eta(\gamma) e^{-\gamma} \big[-\Delta^M(\gamma) -m  \operatorname{Hess}^{\mathbb{R}^{m+1}}_{\gamma} ( \eta , \eta ) + (1-m)|\operatorname{grad}^M \gamma |^2
\\\\- |A|^2+m (1-p) \eta(\gamma)^2\big]
   + \Delta^M(\eta(\gamma) e^{-\gamma}) + (m-2) ( \operatorname{grad}^M \gamma ) (\eta(\gamma) e^{-\gamma} )= 0; \\\\
-2 A(\operatorname{grad}^M( \eta(\gamma) e^{-\gamma})) +2 (1-m) \eta(\gamma) e^{-\gamma} A(\operatorname{grad}^M \gamma) \\\\+ (2p-m)\eta(\gamma)\operatorname{grad}^M (\eta(\gamma) e^{-\gamma}) = 0,
\end{array}
\right.
\end{equation}
where $\operatorname{Hess}^{\mathbb{R}^{m+1}}_{\gamma}$ is the Hessian of the smooth function $\gamma$ in $(\mathbb{R}^{m+1} , h) $.
\end{theorem}

\begin{proof}
By using the Kozul's formula, we have
$$\left\{
  \begin{array}{ll}
    \widetilde{\nabla}_X^M Y  = \nabla_X^M Y + X(\gamma) Y + Y(\gamma)X -g(X,Y) \operatorname{grad}^M \gamma;  \\\\
    \widetilde{\nabla}_U^{\mathbb{R}^{m+1}}  V  = \nabla_U^{\mathbb{R}^{m+1}} V + U(\gamma) V + V(\gamma)U -h(U,V) \operatorname{grad}^{\mathbb{R}^{m+1}} \gamma,
  \end{array}
\right.$$
for all $X,Y \in\Gamma(TM )$, and $ U,V \in \Gamma(T\mathbb{R}^{m+1} )$.
Consequently
\begin{eqnarray}\label{eq3.2}
\nabla_X^{\widetilde{\mathbf{i}}} d\widetilde{\mathbf{i}}(Y)  &=&\nonumber\nabla_{X}^{\widetilde{\mathbf{i}}} Y \\
                                                             &=&\nonumber \widetilde{\nabla}_{d\mathbf{i} (X)}^{\mathbb{R}^{m+1}} Y\\
                                                             &=&\nonumber \widetilde{\nabla}_X^{\mathbb{R}^{m+1}} Y \\
                                                             &=& \nabla_X^{\mathbb{R}^{m+1}} Y + X(\gamma) Y+ Y(\gamma)X -h(X,Y) \operatorname{grad}^{\mathbb{R}^{m+1}} \gamma,\qquad
\end{eqnarray}
and the following
\begin{eqnarray}\label{eq3.3}
d\widetilde{\mathbf{i}} ( \widetilde{\nabla}_X^M Y )
&=&\nonumber d\mathbf{i}(\nabla_X^M Y) + X(\gamma) d\mathbf{i} (Y) + Y(\gamma) d\mathbf{i}(X) - g(X,Y) d\mathbf{i}( \operatorname{grad}^M \gamma )\\
&=&\nabla_X^M Y + X(\gamma) Y + Y(\gamma) X - g(X,Y)  \operatorname{grad}^M \gamma.
\end{eqnarray}
From equations (\ref{eq3.2}) and (\ref{eq3.3}), we get
\begin{eqnarray}\label{eq3.4}
(\nabla d\widetilde{\mathbf{i}})(X,Y)
&=&\nonumber \nabla_X^{\widetilde {\mathbf{i}} } d \widetilde {\mathbf{i}}  (Y)  -d \widetilde {\mathbf{i}}(\widetilde{\nabla}_X^M Y)\\
&=&\nonumber (\nabla d\mathbf{i})(X,Y) + g(X,Y)[\operatorname{grad}^M\gamma -\operatorname{grad}^{\mathbb{R}^{m+1}} \gamma]\\
&=& B(X,Y) -g(X,Y) \eta(\gamma)\eta.
\end{eqnarray}
So that, the mean curvature function $\widetilde{f}$ of $(M^m,\widetilde{g})$  in $ (\mathbb{R}^{m+1} , \widetilde{h} ) $ is given by $ \widetilde{f}=-\eta(\gamma)e^{-\gamma} $. Indeed, by taking traces in (\ref{eq3.4}), we obtain $$e^{2\gamma}\widetilde{H} =  H -  \eta(\gamma)\eta.$$
Since $(M^m,g)$  is minimal in $ (\mathbb{R}^{m+1} , h ) $, we find that $\widetilde{H} =   - e^{-2\gamma} \eta(\gamma)\eta$, that is
$\widetilde{H} =   - e^{-\gamma} \eta(\gamma)\widetilde{\eta}$.\\
With the new notations the equation $(\ref{sys1})$ for $p$-biharmonic hypersurface in the conformally flat space becomes

\begin{equation}\label{sys5}
\left\{
\begin{array}{lll}
-\widetilde{\Delta}(\widetilde{f})  + \widetilde{f} |\widetilde{A}|_{\widetilde{g}}^2 -\widetilde{f} \, \widetilde{\operatorname{Ric}}^{\mathbb{R}^{m+1} }(\widetilde{\eta} , \widetilde{\eta} ) + m(p-2)   \widetilde{f}^3 &=& 0; \\\\
 2\widetilde{A}(\widetilde{\operatorname{grad}}^M \widetilde{f}) -2 \widetilde{f} (\widetilde{\operatorname{Ricci}}^{\mathbb{R}^{m+1} } \widetilde{\eta})^\top + ( p-2 + \dfrac{m}{2} ) \widetilde{\operatorname{grad}}^M \widetilde{f}^2 &=& 0,
\end{array}
\right.
\end{equation}
A straightforward computation yields
\begin{eqnarray*}
\nonumber \widetilde{\operatorname{Ricci}}^{\mathbb{R}^{m+1}} \eta
&=& e^{-2\gamma} \big[\operatorname{Ricci}^{\mathbb{R}^{m+1}}\eta -\Delta^{\mathbb{R}^{m+1}}(\gamma)\eta+(1-m)\nabla_{\eta}^{\mathbb{R}^{m+1}} \operatorname{grad}^{\mathbb{R}^{m+1}} \gamma
\\
\nonumber &&+(1-m) |\operatorname{grad}^{\mathbb{R}^{m+1}} \gamma|^2 \eta
- (1-m) \eta(\gamma)\operatorname{grad}^{\mathbb{R}^{m+1}} \gamma \big];
\end{eqnarray*}
\begin{eqnarray}\label{eq5}
\nonumber \widetilde{\operatorname{Ric}}^{\mathbb{R}^{m+1} }(\widetilde{\eta} , \widetilde{\eta} )
&=& \widetilde{h}(\widetilde{\operatorname{Ricci}}^{\mathbb{R}^{m+1}} \widetilde{\eta},\widetilde{\eta}) \\
\nonumber &=&h(\widetilde{\operatorname{Ricci}}^{\mathbb{R}^{m+1}}\eta,\eta)\\
\nonumber&=& e^{-2\gamma} h\big(\operatorname{Ricci}^{\mathbb{R}^{m+1}}\eta - \Delta^{\mathbb{R}^{m+1}}(\gamma) \eta+(1-m)\nabla_{\eta}^{\mathbb{R}^{m+1}} \operatorname{grad}^{\mathbb{R}^{m+1}} \gamma\\
 \nonumber &&+ (1-m) |\operatorname{grad}^{\mathbb{R}^{m+1}}\gamma|^2 \eta - (1-m) \eta(\gamma)\operatorname{grad}^{\mathbb{R}^{m+1}}\gamma,\eta\big)  \\
\nonumber &=& e^{-2\gamma} \big[-\Delta^{\mathbb{R}^{m+1}}(\gamma) +(1-m) \operatorname{Hess}_{\gamma}^{\mathbb{R}^{m+1}}(\eta,\eta) +(1-m) |\operatorname{grad}^{\mathbb{R}^{m+1}}\gamma|^2\\
&& -(1-m)\eta(\gamma)^2\big];
\end{eqnarray}
\begin{eqnarray}\label{eq8}
 (\widetilde{\operatorname{Ricci}}^{\mathbb{R}^{m+1}} \widetilde{\eta})^\top
 \nonumber&=&\sum_{i=1}^mh(\widetilde{\operatorname{Ricci}}^{\mathbb{R}^{m+1} } \widetilde{\eta} ,e_i)e_i \\
\nonumber &=& (1-m)e^{-3\gamma}\sum_{i=1}^m\left[  h ( \nabla_{\eta}^{\mathbb{R}^{m+1}}\operatorname{grad}^{\mathbb{R}^{m+1}} \gamma , e_i)e_i -\eta(\gamma)h(\operatorname{grad}^{\mathbb{R}^{m+1}} \gamma , e_i)e_i \right]\\
\nonumber &=& (1-m) e^{-3\gamma}\Big[\sum_{i=1}^m h ( \nabla_{e_i}^{\mathbb{R}^{m+1}}\operatorname{grad}^{\mathbb{R}^{m+1}} \gamma , \eta)e_i
              -\eta(\gamma) \operatorname{grad}^M \gamma\Big] \\
\nonumber &=& (1-m) e^{-3\gamma}\Big[\sum_{i=1}^m e_i h(\operatorname{grad}^{\mathbb{R}^{m+1}} \gamma , \eta)e_i
-\sum_{i=1}^m h(\operatorname{grad}^{\mathbb{R}^{m+1}} \gamma  , \nabla_{e_i}^{\mathbb{R}^{m+1}} \eta ) e_i\\
\nonumber  &&- \eta(\gamma) \operatorname{grad}^{M} \gamma\Big]   \\
\nonumber &=& (1-m) e^{-3\gamma} \big[\operatorname{grad}^M \eta(\gamma)
  +\sum_{i=1}^m h(\operatorname{grad}^{\mathbb{R}^{m+1}} \gamma , Ae_i)e_i
 -\eta(\gamma) \operatorname{grad}^M \gamma\big]\\
   &=& (1-m) e^{-3\gamma} \big[\operatorname{grad}^M \eta(\gamma)
+ A(\operatorname{grad}^{M} \gamma) - \eta(\gamma) \operatorname{grad}^M \gamma\big];
\end{eqnarray}
\begin{eqnarray}\label{eq6}
\nonumber \widetilde{\Delta}(\widetilde{f})
&=& e^{-2\gamma}[\Delta(\widetilde{f})+(m-2)d\widetilde{f}(\operatorname{grad}^M \gamma )] \\
&=& e^{-2\gamma}[-\Delta(\eta(\gamma)e^{-\gamma})-(m-2)(\operatorname{grad}^M \gamma)(\eta(\gamma)e^{-\gamma})];
\end{eqnarray}
\begin{eqnarray}\label{eq3.8}
|\widetilde{A}|_{\widetilde{g}}^2
&=&\nonumber\sum_{i=1}^m\widetilde{g}(\widetilde{A}\widetilde{e}_i,\widetilde{A}\widetilde{e}_i)\\
&=&\nonumber\sum_{i=1}^m g(\widetilde{A}e_i,\widetilde{A}e_i)\\
&=&\nonumber\sum_{i=1}^m h(\widetilde{\nabla}_{e_i}^{\mathbb{R}^{m+1}} \widetilde{\eta},\widetilde{\nabla}_{e_i}^{\mathbb{R}^{m+1}} \widetilde{\eta}) \\
&=&\nonumber\sum_{i=1}^m h(\nabla_{e_i}^{\mathbb{R}^{m+1}}\widetilde{\eta}+e_i(\gamma) \widetilde{\eta} +\widetilde{\eta}(\gamma)e_i,\nabla_{e_i}^{\mathbb{R}^{m+1}}\widetilde{\eta}+e_i(\gamma) \widetilde{\eta} +\widetilde{\eta}(\gamma)e_i)\\
&=&\nonumber\sum_{i=1}^m \big[h(\nabla_{e_i}^{\mathbb{R}^{m+1}}\widetilde{\eta},\nabla_{e_i}^{\mathbb{R}^{m+1}}\widetilde{\eta}) +2 \widetilde{\eta}(\gamma)h(\nabla_{e_i}^{\mathbb{R}^{m+1}}\widetilde{\eta},e_i)+e_i(\gamma)^2e^{-2\gamma}\\
&&+2e_i(\gamma)h(\nabla_{e_i}^{\mathbb{R}^{m+1}}\widetilde{\eta},\widetilde{\eta})\big]+m\widetilde{\eta}(\gamma)^2.
\end{eqnarray}
The first term of (\ref{eq3.8}) is given by
\begin{eqnarray*}
\sum_{i=1}^mh(\nabla_{e_i}^{\mathbb{R}^{m+1}}e^{-\gamma}\eta,\nabla_{e_i}^{\mathbb{R}^{m+1}}e^{-\gamma}\eta)
&=&\sum_{i=1}^m h(-e^{-\gamma}e_i(\gamma) \eta+e^{-\gamma}\nabla_{e_i}^{\mathbb{R}^{m+1}}\eta , -e^{-\gamma}e_i(\gamma) \eta+e^{-\gamma}\nabla_{e_i}^{\mathbb{R}^{m+1}}\eta )\\
&=&\sum_{i=1}^m[ e^{-2\gamma}e_i(\gamma)^2+e^{-2\gamma}h(\nabla_{e_i}^{\mathbb{R}^{m+1}}\eta,\nabla_{e_i}^{\mathbb{R}^{m+1}}\eta)]\\
&=& e^{-2\gamma} |\operatorname{grad}^M \gamma|^2 + e^{-2\gamma}|A|^2.
\end{eqnarray*}
The second term of (\ref{eq3.8}) is given by
\begin{eqnarray*}
2\widetilde{\eta}(\gamma)\sum_{i=1}^m h(\nabla_{e_i}^{\mathbb{R}^{m+1}}\widetilde{\eta},e_i)
&=& -2e^{-\gamma} \eta(\gamma)\sum_{i=1}^m  h(e^{-\gamma}\eta,\nabla_{e_i}^{\mathbb{R}^{m+1}} e_i ) \\
&=& -2me^{-2\gamma} \eta(\gamma)h(\eta, H) \\
&=&0.
\end{eqnarray*}
Here $H=0$. We have also
\begin{eqnarray*}
2\sum_{i=1}^me_i(\gamma)h(\nabla_{e_i}^{\mathbb{R}^{m+1}} \widetilde{\eta},\widetilde{\eta})
& =&\sum_{i=1}^m e_i(\gamma) e_i h(\widetilde{\eta} , \widetilde{\eta}) \\
&=& \sum_{i=1}^me_i(\gamma) e_i(e^{-2\gamma}) \\
&=& -2e^{-2\gamma}\sum_{i=1}^m e_i(\gamma)^2 \\
&=& -2e^{-2\gamma} |\operatorname{grad}^M \gamma|^2.
\end{eqnarray*}
Thus
\begin{equation}\label{eq7}
|\widetilde{A}|_{\widetilde{h}}^2= e^{-2\gamma} |A|^2 + m e^{-2\gamma}\eta(\gamma)^2.
\end{equation}
We compute
\begin{eqnarray}\label{eq9}
\nonumber \widetilde{\operatorname{grad}}^M \widetilde{f}&=& e^{-2\gamma} \sum_{i=1}^me_i(\widetilde{f})e_i\\
                                                         &=& -e^{-2\gamma}\operatorname{grad}^M (\eta(\gamma)e^{-\gamma});
\end{eqnarray}
and the following
\begin{eqnarray}\label{eq10}
\nonumber \widetilde{A}(\widetilde{\operatorname{grad}}^M \widetilde{f})
&=& -\widetilde{\nabla}_{\widetilde{\operatorname{grad}}^M \widetilde{f}}^{\mathbb{R}^{m+1}} \widetilde{\eta} \\
\nonumber &=&-\widetilde{\nabla}_{\widetilde{\operatorname{grad}}^M \widetilde{f}}^{\mathbb{R}^{m+1}} e^{-\gamma} \eta \\
\nonumber &=& e^{-\gamma}(\widetilde{\operatorname{grad}}^M \widetilde{f}) (\gamma)  \eta - e^{-\gamma}\widetilde{\nabla}_{\widetilde{\operatorname{grad}}^M \widetilde{f}}^{\mathbb{R}^{m+1}}\eta \\
\nonumber &=& -e^{-3\gamma} \operatorname{grad}^M (\eta(\gamma)e^{-\gamma} ) (\gamma) \eta + e^{-3\gamma} \widetilde{\nabla}_{\operatorname{grad}^M (\eta(\gamma)e^{-\gamma})}^{\mathbb{R}^{m+1}}\eta\\
 \nonumber &=& -e^{-3\gamma}  \operatorname{grad}^M(\eta(\gamma)e^{-\gamma})(\gamma) \eta  + e^{-3\gamma} \eta(\gamma) \operatorname{grad}^M(\eta(\gamma)e^{-\gamma})\\
\nonumber && +e^{-3\gamma}  \operatorname{grad}^M(\eta(\gamma)e^{-\gamma})(\gamma) \eta + e^{-3\gamma}\nabla_{\operatorname{grad}^M(\eta(\gamma)e^{-\gamma})}^{\mathbb{R}^{m+1}} \eta \\
  &=&  e^{-3\gamma} \eta(\gamma) \operatorname{grad}^M(\eta(\gamma)e^{-\gamma}) - e^{-3\gamma } A(\operatorname{grad}^M \eta(\gamma) e^{-\gamma}).\quad\qquad
\end{eqnarray}
Substituting $(\ref{eq5})-(\ref{eq10})$ in $(\ref{sys5})$, and by simplifying the resulting equation we obtain the system $(\ref{sys4})$.
\end{proof}

\begin{remark}\label{remark}\quad
\begin{enumerate}
  \item Using Theorem \ref{th3}, we can construct many examples for proper $p$-biharmonic hypersurfaces in the conformally flat space.
  \item If the functions $\gamma$ and $\eta(\gamma)$ are non-zero constants on $M$, then according to Theorem \ref{th3}, the hypersurface $(M^m , \widetilde{g})$ is $p$-biharmonic in $(\mathbb{R}^{m+1} , \widetilde{h}) $ if and only if
 $$|A|^2=m (1-p) \eta(\gamma)^2-m\eta(\eta(\gamma)).$$
\end{enumerate}
\end{remark}

\begin{example}
The hyperplane $\mathbf{i}  :\mathbb{R}^m \hookrightarrow (\mathbb{R}^{m+1} ,  e^{2\gamma(z)} h ) $, $x \longmapsto(x , c) $, where $ \gamma \in C^{\infty}(\mathbb{R}) $,  $h= \sum_{i=1}^{m}dx_i^2 + dz^2 $, and $c\in \mathbb{R}$,   is proper $p$-biharmonic if and only if
$(1-p)\gamma'(c)^2-\gamma''(c)=0$. Note that, the smooth function
$$\gamma(z)=\frac{\ln\left(c_1(p-1)z+c_2(p-1)\right)}{p-1},\quad c_1,c_2\in\mathbb{R},$$
is a solution of the previous differential equation (for all $c$).
\end{example}

\begin{example}
Let $M$ be a surface of revolution in $\{(x,y,z)\in\mathbb{R}^3\,|\,z>0\}$.
If $M$ is part of a plane orthogonal to the axis of revolution, so that $M$ is parametrized by $$(x_1,x_2)\longmapsto(f(x_2)\cos(x_1),f(x_2)\sin(x_1),c),$$ for some constant $c>0$. Here $f(x_2)>0$. Then, $M$ is minimal, and according to Theorem \ref{th3}, the surface $M$ is proper $p$-biharmonic in
$3$-dimensional hyperbolic space $(\mathbb{H}^3,z^{\frac{2}{p-1}}h)$, where $h=dx^2+dy^2+dz^2$.
\end{example}

\textbf{Open Problems.}
\begin{enumerate}
  \item If $M$ is a minimal surface of revolution contained in a catenoid, that is $M$ is parametrized by $$(x_1,x_2)\longmapsto\left(a\cosh\left(\frac{x_2}{a}+b\right)\cos(x_1),a\cosh\left(\frac{x_2}{a}+b\right)\sin(x_1),x_2\right),$$ where $a\neq0$ and $b$ are constants.
Is there $p\geq2$ and $\gamma\in C^\infty(\mathbb{R}^3)$ such that $M$ is proper $p$-biharmonic in $\left(\mathbb{R}^3,e^{2\gamma}(dx^2+dy^2+dz^2)\right)$?
  \item Is there a proper $p$-biharmonic submanifolds in Euclidean space $(\mathbb{R}^n,dx_1^2+...+dx_n^2)$?
\end{enumerate}

\footnotesize{

}

\end{document}